\documentclass{amsart}
\usepackage[T1]{fontenc}
\usepackage[latin9]{inputenc}
\usepackage[english]{babel}
\usepackage{amssymb}
\usepackage{shltx}
\usepackage{shthm}
\usepackage[sc]{mathpazo}

\makeatletter

\newcommand\sbI{_{\mathrm{I}}}
\newcommand\sbII{_{\mathrm{II}}}

\newcommand\germ{\mathfrak}

\newcommand\newlie[2]{%
  \newcommand#1{{\germ{#2}}}%
}%

\newlie\Gg{g}
\newlie\Gk{k}
\newlie\Gp{p}
\newlie\Gq{q}
\newlie\Ga{a}
\newlie\Gu{u}
\newlie\Gv{v}
\newlie\Gw{w}
\newlie\Gz{z}
\newlie\Gm{m}

\newlie\centralizer{c}
\newcommand\pcentralizer{
  {\centralizer_{\Gp}}%
}
\newcommand\kcentralizer{
  {\centralizer_{\Gk}}%
}
\newcommand\gcentralizer{
  {\centralizer_{\Gg}}%
}
\newcommand\dcentralizer{
  {\centralizer^2_{\Gp}}%
}

\newlie\normalizer{n}

\newlie\Red{R}
\newlie\Ab{A}

\newcommand\newthingie[2]{%
  \begingroup
  \rtA={#2}%
  \rtD={#1}
  \newthingie@A
}

\newcommand\newthingie@A[2][]{%
  \rtB={#1}%
  \rtC={#2}%
  \edef\rtN{\noexpand\newcommand\the\rtA{\the\rtD{\the\rtC}\the\rtB}}%
  \expandafter\endgroup\rtN
}

\let\GO\undefined

\renewcommand\newgroup{\newthingie\mathbf}
\newcommand\newalgebra{\newthingie\germ}

\newgroup\GGL{GL}	\newalgebra\Ggl{gl}
\newgroup\GSL{SL}	\newalgebra\Gsl{sl}
\newgroup\GSp{Sp}	\newalgebra\Gsp{sp}
\newgroup\GSO{SO}	\newalgebra\Gso{so}
\newgroup\GO{O}		\newalgebra\Go{o}
\newgroup\GGO{GO}	\newalgebra\Ggo{go}
\newgroup\GSU{SU}	\newalgebra\Gsu{su}
\newgroup\GU{U}		

\newlie\GG{G}

\newcommand\newscript[2]{%
  \newcommand#1{{\mathcal{#2}}}%
}%

\newscript\Lie{L}
\newscript\Grass{G}
\newscript\Tautological{T}
\newscript\Centralizer{C}
\newscript\Regular{R}
\newscript\Irregular{W}
\newscript\Semisimple{S}
\newscript\Nilpotent{N}
\newscript\Decomposition{D}
\newscript\Orbit{O}
\newscript\HilbertS{H}
\newscript\lhom{L}
\newscript\Struct{O}
\newscript\Incidence{J}
\newscript\Basis{B}
\newscript\Flag{F}

\DeclareMathOperator\rank{rk}
\DeclareMathOperator\ad{ad}

\newcommand\GrassGp{
  {\Grass(r,\Gp)}
}

\newcommand\ProjGp{
  {\P(\Alt^r \Gp)}
}

\newcommand\NilpotentPart[1]{\Nilpotent(#1)}
\newcommand\SemisimplePart[1]{\Semisimple(#1)}
\newcommand\GaN{\NilpotentPart{\Ga}}
\newcommand\GaS{\SemisimplePart{\Ga}}

\let\goodquo=\rightgoodquo

\let\\\newname

\\\BB{Bialynicki-Birula}
\\\BEZOUT{Bezout}
\\\BOREL[A.]{Borel}			
\\\BOURBAKI[N.]{Bourbaki}		
\\\BETTI[E.]{Betti}			
\\\BOHRO{Bohro}
\\\BROER[A.]{Broer}			
\\\BRUHAT[F.]{Bruhat}			
\\\CARTAN[É.]{Cartan}			
\\\CARTIER[P.]{Cartier}			
\\\CHAPUT[P.-E.]{Chaput}		
\\\CHASLES[R.]{Chasles}			
\\\CHEVALLEY[C.]{Chevalley}		
\\\CHERN[S.-S.]{Chern}			
\\\COXETER[H.]{Coxeter}			
\\\DYNKIN[E.]{Dynkin}			
\\\EULER[L.]{Euler}			
\\\FANO[G.]{Fano}			
\\\FUJITA{Fujita}
\\\FULTON[W.]{Fulton}			
\\\FURUSHIMA{Furushima}
\\\GAUSS[K.-F.]{Gauss}			
\\\GRASSMANN{Grassmann}
\\\GROTHENDIECK[A.]{Grothendieck}	
\\\JORDAN[C.]{Jordan}			
\\\KODAIRA{Kodaira}
\\\HARRIS[J.]{Harris}			
\\\HASSE{Hasse}
\\\HILBERT[D.]{Hilbert}			
\\\HIRZEBRUCH[F.]{Hirzebruch}		
\\\HUMPHREYS[J.-E.]{Humphreys}		
\\\ILIEV[A.]{Iliev}			
\\\IWASAWA{Iwasawa}
\\\KILLING{Killing}
\\\KLEIN[F.]{Klein}			
\\\KOLCHIN{Kolchin}
\\\KOSTANT[B.]{Kostant}			
\\\KNOP[F.]{Knop}
\\\KRAFT[H.]{Kraft}			
\\\KRONECKER[L.]{Kronecker}		
\\\LANDAU{Landau}
\\\LANDSBERG[J.]{Landsberg}		
\\\LAURENT{Laurent}
\\\LEFSCHETZ{Lefschetz}
\\\LEPOWSKY{Lepowsky}
\\\LEVI{Levi}
\\\LIE[S.]{Lie}				
\\\MALCEV[A.~I.]{Malcev}
\\\MANIVEL[L.]{Manivel}			
\\\MUMFORD[D.]{Mumford}			
\\\MUKAI[S.]{Mukai}
\\\OHTA[T.]{Ohta}
\\\PANYUSHEV[D.]{Panyushev}		
\\\PICARD{Picard}
\\\PLUCKER{Plücker}
\\\POINCARE[R.]{Poincaré}		
\\\PROCESI[C.]{Procesi}
\\\RALLIS[S.]{Rallis}
\\\RANESTAD[K.]{Ranestad}		
\\\RICHARDSON[R.~W.]{Richardson}
\\\RIEMANN[B.]{Riemann}			
\\\ROCH{Roch}
\\\ROSENLICHT{Rosenlicht}
\\\SCHREYER[F.-O.]{Schreyer}		
\\\SCHUR[I.]{Schur}			
\\\SCORZA{Scorza}
\\\SERRE[J.-P.]{Serre}			
\\\SEVERI[F.]{Severi}			
\\\SMITH{Smith}
\\\SMS[S.]{Müller-Stach}		
\\\TAUVEL[P.]{Tauvel}			
\\\VDV{van de Ven}
\\\VELDKAMP[F. D.]{Veldkamp}
\\\VERONESE{Veronese}
\\\VUST[T.]{Vust}			
\\\WEIL{Weil}
\\\WEYL{Weyl}
\\\WITT{Witt}
\\\YU[R.]{Yu}				
\\\YOUNG{Young}
\\\WHITNEY{Whitney}
\\\ZARISKI[O.]{Zariski}			

\renewcommand\nameprepare[2]{#2}

\def\IM{\ILIEV\ and~\MANIVEL}
\def\TY{\TAUVEL\ and~\YU}

\def\RR{Riemann-Roch}

\def\CJ{\CHEVALLEY-\JORDAN}

\newcommand\refcite[2]{\cite[#1]{#2}}

\dblbkslashrestore
\makeatother

\begin{document}
\author{Michaël Le Barbier Grünewald}
\address{%
Michaël Le Barbier Grünewald\\
Max-Planck-Institut für Mathematik\\
Vivatsgasse 7\\
$53\,111$ Bonn
}
\email{michi@mpim-bonn.mpg.de}
\urladdr{http://guests.mpim-bonn.mpg.de/michi}
\title{The variety of reductions for a reductive symmetric pair}
\date{18 December 2008}



\begin{abstract}
We define and study the variety of reductions for a reductive symmetric
pair~$(G,\theta)$, which is the natural compactification of the set of
the~\CARTAN\ subspaces of the symmetric pair. These varieties generalize 
the varieties of reductions for the \SEVERI\ varieties studied by \ILIEV\ and
\MANIVEL, which are \FANO\ varieties.

We develop a theoretical basis to the study these varieties of reductions, and
relate the geometry of these variety to some problems in representation
theory. A very useful result is the rigidity of semi-simple elements in
deformations of algebraic subalgebras of \LIE\ algebras.

We apply this theory to the study of other varieties of reductions in a
companion paper, which yields two new \FANO\ varieties.
\end{abstract}


\maketitle



%

\section{Introduction}

The problem of classifying all complex analytic compactifications of~$\C^n$
which have second \BETTI\ number~$b_2 = 1$, also known as~\emph{irreducible
compactifications}, was stated by \HIRZEBRUCH~\cite{HIRZEBRUCH}.
An~irreducible compactification is always
a \FANO\ variety, and classifiying the former ones is, surprisingly,
not easier than classifying the later ones. Indeed, the classifications of the
irreducible compactifications of dimension~$1$, $2$ and~$3$ came as
specializations of the classifications of the~\FANO\ varieties of the
corresponding dimension. For the dimension~$4$ or above, there is at the
current time no classification of the irreducible compactifications available
and no classification of the~\FANO\ varieties.  The
irreducible compactifications of the affine spaces of dimension less than~$3$
are the projective spaces~$\P^1$, $\P^2$ and~$\P^3$, the quadric~$\Q^3$, and
two \FANO\ varieties which have~$b_3 = 0$ and respective \FANO\ index~$1$
and~$2$. See~\SMS~\cite{SMS1} for a short description of these
varieties and an account on other results about compactifications of affine
spaces.

As pointed out in~\cite{HIRZEBRUCH}, the fundamental homogenous spaces of
reductive groups are examples of compactifications of the affine space, but
new examples are very hard to find and seem to always show up as finite
families. This lack of examples hinders the efforts aimed at the
classification of \FANO\ varieties or even that of irreducible
compactifications of affine spaces. In this paper we associate to any
\emph{reductive symmetric pair}~$(G,\theta)$, where~$G$ is a reductive complex
group and~$\theta$ an involution of~$G$, its \emph{variety of
reductions}~$\Red$. Some of these varieties were previously obtained by
different means, and studied by~\RANESTAD\ and~\SCHREYER~\cite{RS} and~\ILIEV\
and~\MANIVEL~\cite{ILR}.  At the present time, the study of~seven of these
varieties of reductions~$\Red$ has been carried out, revealing that all of
them are normal \FANO\ varieties, and the smooth ones are even
compactifications of the affine space which have~$b_2 = 1$. This observation
is our main motivation for defining the varieties of reductions:
the study of many of the low-dimensional ones may be carried out, and
yield other compactifications of affine spaces or \FANO\ varieties.

We study here general properties of varieties of reductions for symmetric
pairs, and use this theory in a companion
paper~\cite{COMPANION} to study three more examples of variety of
reductions~(two of them count in the number~seven mentioned above). While we
were primarily interested in developping tools and methods suited to the
practical study of examples, it turned out that the general theory presents
interesting aspects in its own right and relates to other problems in the
representation theory of complex \LIE\ groups. We now define the varieties of
reductions and then outline the results of our study.

\subsubsection{Variety of reductions for a symmetric pair}
A~\emph{reductive symmetric} pair~$(G,\theta)$ has a reductive group~$G$ as
its first member and an involution~$\theta$ of~$G$ as its second member. These
pairs occur in the study of real forms of complex reductive groups and
symmetric spaces, and they were classified using many different invariants,
see S.~Araki~\cite{ARAKI}, A.~G.~Helminck~\cite{HELMINCK},
and~T.~Springer~\cite{SPRINGER}, for instance. To such a pair we attach its
\emph{connected fixed point group}~$K = (G^\theta)^\circ$, which is reductive,
and the
decomposition of the~\LIE\ algebra~$\Gg$ of~$G$ as eigenspaces for the
involution~$\theta'$ tangent to~$\theta$ at the unit element of~$G$:
\[
\Gg = \Gg(\theta')_{1} \oplus \Gg(\theta')_{-1} = \Gk \oplus \Gp
.
\]
This decomposition is called the~\CARTAN\ decomposition of~$\Gg$, and~$\Gp$ 
the \emph{anisotropic space} of the symmetric pair. We will use~\cite{KR}
by~\KOSTANT\ and~\RALLIS\ and~\cite{TY} by~\TY\ as references for results
about the operation of~$K$ in the anisotropic space~$\Gp$.

\begin{definition}
A~\emph{\CARTAN\ subspace} of~$\Gp$ is a linear subspace~$\Ga$ of~$\Gg$ that
is contained in~$\Gp$ and in some \CARTAN\ subalgebra of~$\Gg$, and that is
maximal in the family of such subspaces ordered by the inclusion.
\end{definition}

Any two~\CARTAN\ subspaces of~$\Gp$
are~$K$-conjugated~\cite[Theorem~1]{KR}. Their common
dimension~$r$ is the~\emph{rank} of the symmetric pair~$(G,\theta)$ and their
set~$\Red_o$ is a~$K$-orbit in the \GRASSMANN\ variety~$\Grass(r,\Gp)$
of~$r$-planes in~$\Gp$. 

\begin{definition}
\label{ss:sp:defred}
The variety of reductions~$\Red$ for the symmetric pair~$(G,\theta)$ is the
closure in~$\Grass(r,\Gp)$ of the~set~$\Red_o$ of all \CARTAN\ subspace
of~$\Gp$.
\end{definition}

It is customary, while sometimes ambiguous, to write~$(G,K)$
for~$(G,\theta)$ when referring to a particular symmetric pair. We stick
to this usage, and emphasize that the Cartesian square~$\GG
\times \GG$ of a reductive group~$\GG$ is turned into a symmetric pair by
the involution swapping its two factors. \RANESTAD\ and~\SCHREYER~\cite{RS} have
shown that the variety of reductions for the symmetric pairs~$(\SL_n,\SO_n)$
are smooth only for~$n \le 5$. \ILIEV\ and~\MANIVEL~\cite{ILR} studied the
varieties of reductions for the symmetric pairs~$(\SL_3,\SO_3)$,
$(\SL_3\times\SL_3,\SL_3)$, $(\SL_6,\Sp_6)$ and~$(E_6,F_4)$. These four
symmetric pairs occur as structure symmetries for the four simple \JORDAN\ %
algebras of rank~$3$~\cite{ILR}. In~\cite{COMPANION} we study the varieties of
reductions for~$(\SL_4,\SO_4)$, and for the Cartesian squares of~$\Sp_4$
and~$G_2$.

\subsubsection{Abelian subalgebras}
In the small rank cases~\cite{ILR,ILRGLN,COMPANION} the variety~$\Ab$ of all
$r$-dimensional subalgebras of~$\Gg$ contained in~$\Gp$ is not
larger than~$\Red$, but we show the

\begin{theorem*}[\ref{co:sp:algred}, \ref{pr:sp:aanonalg}]
Every point in~$\Red$ is the~\LIE\ algebra of a subgroup of~$G$. If~$G$ has
large enough rank, then~$\Ab$ contains a point that is not the~\LIE\ algebra of
a subgroup of~$G$.
\end{theorem*}
In general~$\Red$ is a strict irreducible component of~$\Ab$.
We also show that~$\Ab$ consists of infinitely many orbits, while
there is still no evidence that the same can happen for~$\Red$. 

Abelian subalgebras
of~$\Gg$ have been extensively studied, by \SCHUR, \MALCEV, \PANYUSHEV\ and
others, but very little is known about the geometry of~$\Ab$. An enumeration of
its irreducible components is not even at hand.  Further
investigations may confirm that the study of~$\Red$ is easier than the one
of~$\Ab$~is.

\subsubsection{Rigidity of anisotropic subtori}
One of our main result~(\ref{co:sp:ar}) states that, if~$\Ga_0$ is the
degeneration of the conjugates of a subalgebra~$\Ga_1$ of~$\Gg$ contained
in~$\Gp$ under the operation of~$K$, then the semi-simple elements of~$\Ga_0$
are rigid. This means that they are the limits of the semi-simple elements in
the conjugates of~$\Ga_1$. This rigidity theorem enables us to study varieties
of reductions through their~\emph{subvarieties of reductions} (see below), but
also to contribute to the theory of decomposition classes. Decomposition
classes for a reductive symmetric pair generalize \JORDAN\ types
for~$\GL_n$. They were introduced by \BOHRO\ and \KRAFT~\cite{BOHROKRAFT} to
study sheets in
\LIE\ algebras. As an application of our rigidity theorem, we show the

\begin{proposition*}[\ref{th:sp:dcclosure}]
The closure of a decomposition class is a union of decomposition classes.
\end{proposition*}

\subsubsection{Orbit theory for the varieties of reductions}
While the theory is still incomplete, there is some open subset of~$\Red$
whose orbits we understand well, by comparing them to
\emph{decomposition classes} in~$\Gg$.

Sending a point~$x\in \Gg$ to its
centralizer, we define a rational map~$\Centralizer$ from~$\Gg$ to~$\Red$.

\begin{proposition*}[\ref{th:sp:redregular}, \ref{pr:sp:dcgen}, \ref{co:sp:dcgen}]
The map~$\Centralizer$ enjoys the following properties:
\begin{enumerate}

\item
Its pointwise image is an open subset~$\Red_r$ of the
smooth locus of~$\Red$.

\item
The pre-image of an orbit~$\Orbit_0$ in~$\Red_r$ is a
decomposition class~$\Decomposition_0$ in~$\Gg$.

\item
If~$\Decomposition_1$ is a decomposition class in~$\Gg$
containing~$\Decomposition_0$ in its closure, then~$\Centralizer$ is defined
at any point of~$\Decomposition_1$ and the image of~$\Decomposition_1$ is an
orbit containing~$\Orbit_0$ in its closure.  
\end{enumerate}
\end{proposition*}
The problem of describing the genericity relations between decomposition
classes is much easier than the analogous problem for orbits.

The previous proposition describes a bunch of orbits of low-codimension
in~$\Red$, but two important questions remain:
What is the codimension of the complement
of the image of~$\Centralizer$? How intricated are the combinatorics of the
orbits in this complement? The variety of reductions
for~$(\SL_4,\SO_4)$ is smooth, while the centralizer map is not surjective. In
the varieties of reductions for the Cartesian squares of~$\Sp_4$ and~$G_2$,
the image of the centralizer map equals the smooth locus, and its complement
has codimension~$2$.

While we are not yet able to answer these questions,
we noticed an interesting structure in the family of varieties of reductions,
described in terms of \emph{subvarieties of reductions}.

Let~$\Ga$ be a \CARTAN\ subalgebra of~$\Gg$,~$\Ga'$ a subalgebra of~$\Ga$
and~$G'$ its centralizer. Let~$\Red'$ denote the closure of~$G'\Ga$
in~$\Red$. As the data~$(\Ga, \Ga')$ runs through all its possible values,
$\Red'$ describes the set of subvarieties of reductions of~$\Red$. Note that
if~$G'$ is smaller than~$G$, the variety~$\Red'$ is isomorphic to the variety
of \CARTAN\ reductions for the derived group of~$G'$, which has smaller rank
than~$G$.  We can now state our

\begin{theorem*}[\ref{pr:sp:srdescr} and its corollaries]
Let~$x_0$ be a point of~$\Red$. Then:
\begin{enumerate}

\item
$x_0$ is
contained in a strict subvariety of reductions of~$\Red$ if, and only if,
it contains a non-nilpotent element of~$\Gg$.

\item
A point~$x_1$ is more general than~$x_0$ in~$\Red$ if, and only if, any
subvariety of reductions~$\Red'$ containing~$x_0$ also contains
a~$G$-conjugate of~$x_1$ that is more general than~$x_0$ in~$\Red'$.
\end{enumerate}
\end{theorem*}

Hence, if we are able to describe orbits in all the subvarieties of reductions
of~$\Red$, and there is only a finite number of isomorphism classes of them,
we can as well describe orbits in the open
subset of~$\Red$ whose complement is
\[
\Red_n = \set { \Gu \in \Red \mid \text{every $u \in \Gu$ is nilpotent } }
.
\]
\begin{proposition*}[\ref{pr:sp:algclosed}]
The subvariety~$\Red_n$ of~$\Red$ contains the closed orbits
of~$\Red$.  
\end{proposition*}
We may learn soon how to enumerate the closed orbits of~$\Red$,
but the detailed orbit theory of~$\Red_n$
remains very mysterious. In particular we do not know whether it can contain
infinitely many~$K$-orbits or not.

\subsubsection{Partial positivity of the anticanonical class}
The minimal rational curves in~$\Red$ are the lines of the natural projective
embedding of~$\Red$ that are contained in~$\Red$. We can describe the generic
ones in
terms of the roots of the operation of~$\Ga$ on~$\Gg$. This description is
precise enough to let us study the deformations of such a line, find
explicit free curves contained in the smooth locus of~$\Red$ and compute
the intersection of the anticanonical class on these lines:
\begin{corollary*}[\ref{co:sp:pphsmooth}]
Let $\Ga\in\Red_o$ be a general point of~$\Red$ and~$\Delta$ a line through~$\Ga$
contained in~$\Red$. Let~$m$ be the dimension of the maximal linear subspace
of~$\Red$ through~$\Ga$ containing~$\Delta$. If~$\Delta$ is contained in the
smooth locus of~$\Red$, then
\[
  - K_r \cdot \Delta = m + 1
\]
where~$K_r$ is the canonical class of the smooth locus of~$\Red$.
\end{corollary*}
This intersection number is always
positive, it equals~$3$ when~$\Red$ is a variety of~\CARTAN\ reductions.
When~$\Red$ has \PICARD\ number one and its canonical class is a
\CARTIER\ divisor, this yields the \FANO\ index of~$\Red$.

The open orbit in~$\Red$ is affine since the stabilizer of a point therein is
reductive, its complement is therefore a union of divisors. In the examples
we studied~\cite{COMPANION}, the complement of the
image~$\Red_o$ of~$\Centralizer$ has codimension at least~$2$---the codimension~$2$
occuring for the variety of \CARTAN\ reductions for~$G_2$. This means that the
image of the centralizer map carries enough information to describe the
\PICARD\ group of~$\Red$. In the general theory, it
is possible to bound by
above the dimension of the orbit of a point outside of~$\Red_o$, but narrowing
our attention to~$\Red_o$ we would miss a divisor in~$\Red\setminus\Red_o$
swept out by a continuous family of small dimensional orbits, if such a divisor
exists. Hence the flaws in our orbit theory of~$\Red$ obstructs
our understanding of the \PICARD\ group of~$\Red$.


\tableofcontents



\section{Variety of reductions for a symmetric pair}
\label{se:sp:defred}

\subsubsection{Variety of reductions}
Let~$(G,\theta)$ be a reductive symmetric pair with rank~$r$ and~$\GrassGp$ the
\GRASSMANN\ variety of $r$-planes in~$\Gp$.

\begin{definition}
We call the set~$\Red_o$ of \CARTAN\ subspaces of~$\Gp$ the set of
\emph{ordinary reductions} for the symmetric pair~$(G,\theta)$. Its
closure~$\Red$ in~$\GrassGp$ is the~\emph{variety of reductions}
for~$(G,\theta)$, and~$\Red_s = \Red\setminus \Red_o$ is the set of
special reductions.
\end{definition}

The variety of reductions depends only of the isogeny class of~$G$,
not of its fundamental group.

Recall that \GRASSMANN\ varieties are embedded in
projective spaces of some exterior powers, in particular a
reduction~$\Ga\in\Red$ of which~$a\in\Ga^r$ is a basis has
image~$[a_1\wedge\cdots\wedge a_r]\in\ProjGp$.

\subsubsection{Variety of anisotropic subalgebras}
\label{ss:sp:defab}
The variety~$\Red$ of reductions for a reductive symmetric pair is a
subvariety of the variety~$\Ab$ of all anisotropic subalgebras of~$\Gg$ of
dimension~$r$.

Let~$A_r
: \Alt^r \Gp \to \Gk\otimes\Alt^{r-2}\Gp$ be the linear map whose
value on the decomposable~$r$-multivector~$x_1 \wedge \cdots \wedge x_r$ is
\[
A_r(x_1 \wedge \cdots \wedge x_r) =
\sum_{1\le i < j\le l} [x_i,x_j]\otimes
  x_1 \wedge\cdots
      \wedge\hat x_i 
      \wedge\cdots
      \wedge\hat x_j
      \wedge\cdots
      \wedge x_r
\]
where terms with a hat above shall be discarded. The kernel~$\ker A_r$
of~$A_r$ is a linear subspace of~$\ProjGp$
meeting~$\GrassGp$ along the set of subalgebras of~$\Gg$
of dimension~$r$ contained in~$\Gp$.

\begin{definition}
The variety of anistropic subalgebras for the
reductive symmetric pair~$(G,\theta)$ is the section~$\Ab$
of~$\GrassGp$ with~$\ker A_r$ in~$\ProjGp$.
\end{definition}

When the symmetric pair associated to a reductive group is under
consideration, we call~$\Ab$ the variety of abelian subalgebras for
the given group.
We illustrate the presumable complexity of~$\Ab$ in
section~(\ref{se:sp:aa}).

\subsubsection{Roots relative to an ordinary reduction}
Let~$\Ga$ be an ordinary reduction and let~$\Phi$ be the set of weights of~$\Ga$
in~$\Gg$. The decomposition of~$\Gg$ into weight spaces relative to~$\Ga$ is
\begin{equation}
\label{eq:sp:spweight}
  \Gg =
    \gcentralizer(\Ga) 
    \oplus
    \bigoplus_{\alpha\in\Phi} \Gg(\Ga)_\alpha
\end{equation}
where~$\gcentralizer(\Ga)$ is the centralizer of~$\Ga$ in~$\Gg$
and~$\Gg(\Ga)_\alpha$ the weight space for the character~$\alpha$ of~$\Ga$.
The following propositions enable us to work with the
decomposition~(\ref{eq:sp:spweight}), for a proof, see~\TY~\cite[37.5.3, 36.2.1,
38.2.7, 38.7.2]{TY}.

\begin{proposition}
\label{pr:sp:spcentralizer}
The centralizer of~$\Ga$ in~$\Gg$ is the direct sum of~$\Ga$ and the
centralizer of~$\Ga$ in~$\Gk$.
\end{proposition}

\begin{proposition}
The set~$\Phi$ is a root system.
\end{proposition}

We emphasize that this root system may not be reduced, however the root system
relative to the symmetric pair of a reductive group is the root system of the
reductive group itself, and is thus always reduced.

Since~$\theta'$ swaps~$\Gg(\Ga)_\alpha$ and~$\Gg(\Ga)_{-\alpha}$, the \CARTAN\
decomposition and the weight spaces decomposition are not
comparable. It is thus useful to introduce the subspaces
\begin{equation}
  \Gp(\Ga)_\alpha = (\Gg(\Ga)_{\alpha} \oplus \Gg(\Ga)_{-\alpha})\cap\Gp
  \quad\mbox{and}\quad
  \Gk(\Ga)_\alpha = (\Gg(\Ga)_{\alpha} \oplus \Gg(\Ga)_{-\alpha})\cap\Gk
  .
\end{equation}
These subspaces satisfy the relations~$\Gp(\Ga)_\alpha = \Gp(\Ga)_{-\alpha}$,
$\Gk(\Ga)_\alpha = \Gk(\Ga)_{-\alpha}$, and~$\Gp(\Ga)_\alpha \oplus
\Gk(\Ga)_{\alpha} = \Gg(\Ga)_{\alpha}\oplus\Gg(\Ga)_{-\alpha}$. We can then
decompose~$\Gk$ and~$\Gp$ in the following manner:
\begin{equation}
\label{eq:sp:spdecoiso}
  \Gk = \germ{m} \oplus \bigoplus_{\alpha\in\Phi^+} \Gk(\Ga)_{\alpha}
\end{equation}
where~$\germ{m} = \kcentralizer(\Ga)$ is the reductive \LIE\ algebra
centralizing~$\Ga$ in~$\Gk$;
\begin{equation}
\label{eq:sp:spdecoaniso}
  \Gp = \Ga \oplus \bigoplus_{\alpha\in\Phi^+} \Gp(\Ga)_{\alpha}
  .
\end{equation}
The spaces~$\Gk(\Ga)_\alpha$ and~$\Gp(\Ga)_\alpha$ are not stable under~$\Ga$ but
are swapped by it: we have $\ad(\Ga)\Gk(\Ga)_\alpha = \Gp(\Ga)_\alpha$
and~$\ad(\Ga)\Gp(\Ga)_\alpha = \Gk(\Ga)_\alpha$. This allows us to compute the
dimension of~$\Red$:

\begin{proposition}
\label{pr:sp:reddim}
The dimension of~$\Red$ is~$\dim\Gp - r$.
\end{proposition}





\section{Regular orbits and the incidence diagram}
\label{se:sp:orcd}

\subsubsection{Incidence variety}
\label{ss:sp:incidence}
Let~$\Incidence$ be the incidence variety associated to~$\Red$:
\[
   \Incidence = \set{ (\Gu,u)\in \Red\times\Gp \mid u \in \Gu }
   .
\]
The projections of~$\Incidence$ to~$\Gp$ and~$\Red$ are respectively
denoted by~$\pi$ and~$\tau$. Notice that, since the tautological
fibre bundle on~$\GrassGp$ is locally trivial, the morphism~$\tau$
is open.

\subsubsection{Centralizer map}
An element~$u\in\Gp$ is \emph{regular} when the dimension of its
centralizer~$\Centralizer(u)$ in~$\Gp$ is minimal. It is then an~$r$-dimensional
subalgebra of~$\Gg$ and~the map~$\Centralizer$
is a well-defined morphism from~$\Regular$ to~$\Ab$.
We call~\emph{regular} a reduction in~$\Red_r = \Centralizer(\Regular)$, or a
$K$-orbit through a regular reduction. 

We write~$\Semisimple$ for the set of semi-simple elements in~$\Gp$,
and~$\Nilpotent$ for the set of nilpotent~ones. The set~$\Regular_o = \Regular
\cap \Semisimple$ of \emph{ordinary elements} in~$\Gp$ that are both regular
and semi-simple is an open subset of~$\Regular$, dense in~$\Gp$.
The set of~$r$-dimensional subspaces of~$\Gp$ containing an ordinary element
is an open subspace of~$\GrassGp$. This implies the following

\begin{proposition}
\label{pr:sp:redirreducible}
The variety~$\Red$ of reductions for a reductive symmetric pair is an irreducible
component of the variety~$\Ab$ of anisotropic subalgebras for this pair.
\end{proposition}

\begin{corollary}
The centralizer of any regular element is a reduction, that
is,~$\Centralizer$ maps~$\Regular$ to~$\Red$.
\end{corollary}

\subsubsection{Incidence diagram}
The incidence variety and the centralizer map fit together in
the~\emph{incidence diagram}:
\[
\begin{matrix}
& \Incidence &\\
\swarrow & & \searrow\\
\usualsetfont{P}(\Gp)& \dashrightarrow & \Red
\end{matrix}
\]
and the centralizer map~$\Centralizer$ equals~$\tau\circ\inv\pi$
above~$\Regular$. We use these constructions to explore the variety of
reductions.

\subsubsection{Regular orbits are contained in the smooth locus}

The centralizer map parametrizes an open and smooth subset of~$\Red$:

\begin{theorem}
\label{th:sp:redregular}
The set~$\Red_r = \Centralizer(\Regular)$ of regular reductions is an open
subset of~$\Red$, contained in the smooth locus.
\end{theorem}

\begin{proof}
The subset of~$\Grass(r,\Gp)$ of subspaces meeting~$\Regular$ is open, which
implies the truth of the first statement. We now study the smoothness.

Let~$u_0\in\Regular$, we study the tangent space to~$\Ab$ at~$\Gu =
\Centralizer(u_0)$. We choose a supplementary~subspace~$\Gv$ to~$\Gu$ in~$\Gp$
and identify the affine neighbourhood of~$\Gu$ in the \GRASSMANN\ %
variety~$\Grass(r,\Gp)$ consisting of points admitting~$\Gv$ for
supplementary with the space~$A = \lhom(\Gu,\Gv)$ of linear maps
from~$\Gu$ to~$\Gv$. A point~$a\in A$ belongs to~$\Ab$ if,
and only if for all~$(u_1,u_2)\in\Gu^2$
\[
  [ u_1 + a (u_1), u_2 + a(u_2) ] = 0;
\]
the linear equations of~$T_\Gu\Ab$ are therefore
\begin{equation}
\label{eq:sp:redtspace}
  [ u_1, a(u_2) ] = [ u_2, a(u_1) ]
\end{equation}
for all~$(u_1,u_2)\in\Gu^2$. We narrow our attention to the subset of the
linear equations of~$T_\Gu\Ab$ obtained by letting~$u_2 = u_0$.
Note that~$\Gu$ is precisely the kernel of the
restriction of~$\ad(u_0)$ to~$\Gp$, hence, once~$u_1$ is given there
is at most one value of~$a(u_1)$ satisfying~(\ref{eq:sp:redtspace}). 
The tangent space is thus
parametrised by a subspace of the set of images~$a(u_0)$, when~$a$ varies
in~$A$ and $T_\Gu\Ab \subset \Gv$. But~$\dim T_\Gu\Ab \ge \dim_\Gu\Ab \ge
\dim\Red$ and~$\dim\Red = \dim\Gv$~(\ref{pr:sp:reddim}). We conclude
that~$\dim T_\Gu\Ab = \dim_\Gu\Ab = \dim\Red$. The regular reductions
are thus contained in the smooth locus of~$\Red$.
\end{proof}

\subsubsection{Review of decomposition classes}
\label{ss:sp:dc}
The centralizer map puts the~$K$-orbits in its image in correspondence
with decomposition classes in~$\Red$. Decomposition classes generalize
the familiar \JORDAN\ types in~$\Gsl_n$ to every symmetric pairs.
They are discussed by~\BOHRO\ and~\KRAFT~\cite{BOHROKRAFT},
\BROER~\cite{BROER} and~\TY~\cite{TY}. We briefly review elements of
this discussion, referring to~\cite{TY} for further details. 

\begin{definition}
Let~$u$ and~$v$ be two elements in~$\Gp$ with
respective~\CHEVALLEY-\JORDAN\ decomposition $u = u_s + u_n$ and~$v = v_s +
v_n$. They have the same \emph{decomposition class} when there exists~$g \in
K$ such that~$g\cdot \pcentralizer(u_s) = \pcentralizer(v_s)$ and~$g
u_n = v_n$. 
\end{definition}
This defines an equivalence relation on~$\Gp$. We
write~$\Decomposition$ for the associated partition
and~$\Decomposition(u)$ for the element of~$\Decomposition$
containing~$u$. Notice that decomposition classes are punctured cones
in~$\Gp$, so that their projectivizations define a partition of
the projective space~$\P(\Gp)$.

This relation is equivalently described with~\emph{double centralizers}. The
double centralizer of~$u\in\Gp$ is
\[
  \dcentralizer(u) = \set{ v \in\Gp \mid
    \forall w \in \pcentralizer(u) \; [v,w] = 0
  }
  .
\]
This is the set of elements in~$\Gp$ whose centralizer in~$\Gp$
contains~$\pcentralizer(u)$. \CHEVALLEY's theorem implies that
\begin{equation}
\label{eq:sp:dcdcentgen}
  \dcentralizer(u)_o = \set{ v \in\Gp \mid
    \pcentralizer(v) = \pcentralizer(u)
  }
\end{equation}
is a dense open subset of~$\dcentralizer(u)$.

\begin{proposition}[\refcite{38.8.3, 39.5.1 and~39.5.4}{TY}]
\label{pr:sp:dcdcent}
For two elements in~$\Gp$, the following statements are equivalent.
\begin{enumerate}
\item
They have the same decomposition class.
\item
Their centralizers in~$\Gp$ are~$K$-conjugate.
\item
Their double centralizers in~$\Gp$ are~$K$-conjugate.
\end{enumerate}
\end{proposition}

\begin{corollary}[\refcite{39.5.5}{TY}]
\label{co:sp:dclc}
The decomposition class~$\Decomposition(u)$ of~$u\in\Gp$ is~$K \cdot
\dcentralizer(u)_o$. This is an irreducible locally closed subvariety
of~$\Gp$.
\end{corollary}

\begin{proposition}[\refcite{39.5.6}{TY}]
The set~$\Decomposition$ of decomposition classes in~$\Gp$ is finite.
\end{proposition}

Thus, decomposition classes form a finite partition of~$\Gp$ into
locally closed sets. It follows that any irreducible subvariety~$X$ of~$\Gp$
or~$\P(\Gp)$ is the
closure of~$X\cap\Decomposition_X$ in~$X$, for a unique decomposition
class~$\Decomposition_X\in\Decomposition$. We say~$\Decomposition_X$ is the
\emph{dominant decomposition class} in~$X$.

If~$u$ is regular, semi-simple or nilpotent, elements of~$\Decomposition(u)$
are regular, semi-simple or nilpotent, and we call~$\Decomposition(u)$ a
regular, semi-simple or nilpotent decomposition class.

In~\ref{th:sp:dcclosure} we use our rigidity theorem~\ref{co:sp:ar} to
prove that the closure of a decomposition class is a union of decomposition
classes.

\subsubsection{Genericity relation for regular orbits}
We want to describe the genericity relation for orbits in~$\Red$, that is,
which orbits lie in the closure of which. 
An orbit~$\Orbit_1$ is \emph{more general} than an orbit~$\Orbit_0$
when~$\Orbit_0$ lies in the closure of~$\Orbit_1$; similarly, a decomposition
class~$\Decomposition_1$ is more general than a decomposition
class~$\Decomposition_0$ when~$\Decomposition_0$ is contained in the closure
of~$\Decomposition_1$.

Through the centralizer map, the genericity relation for regular orbits
reduces to the genericity relation for decomposition classes.

\begin{proposition}
\label{pr:sp:dcregular}
The image of a regular decomposition class through the centralizer map is
a~regular~$K$-orbit of~$\Red$.
\end{proposition}

This is a reformulation of the equivalence between~(1) and~(2)
in~(\ref{pr:sp:dcdcent}). We are then allowed to speak of the
decomposition class associated with a regular orbit: it is its inverse image
through the centralizer map.

\begin{proposition}
\label{pr:sp:dcgen}
Let~$\Red\sbI$ and~$\Red\sbII$ be orbits in~$\Red$. If~$\Red\sbI$ is more
general than~$\Red\sbII$ and~$\Red\sbII$ is regular, then~$\Red\sbI$ is
regular and its decomposition class is more general than the one
of~$\Red\sbII$.
\end{proposition}

\begin{proof}
Let~$\Gp\sbI$ be the projection~$\tau(\inv\pi(\Red\sbI))$
of~$\Red\sbI$ through the incidence diagram
of~$\Red$~(\ref{ss:sp:incidence}), and~$\Gp\sbII$ 
the projection 
of~$\Red\sbII$. Note that~$\Red\sbII$ being regular,~$\Gp\sbII$ is a
decomposition class and is thus irreducible~(\ref{co:sp:dclc}).

The morphism~$\pi$ is continuous, so that~$\inv\pi(\Red\sbI)$
contains~$\inv\pi(\Red\sbII)$ in its closure; the
morphism~$\tau$ is open, so 
that~$\Gp\sbI = \tau(\inv\pi(\Red\sbI))$ contains~$\Gp\sbII$ in its
closure. Let~$\Gq$ be an irreducible component in~$\Gp\sbI$ degenerating
onto~$\Gp\sbII$. The decomposition class~$\Decomposition_\Gq$ dominant
in~$\Gq$ is more general than~$\Gp\sbII$, it is thus regular and~$\Red\sbI$ is
the regular~$K$-orbit~$\Centralizer(\Gq)$ in~$\Red$.
\end{proof}

\begin{corollary}
\label{co:sp:dcgen}
Two regular orbits compare in the same way than their corresponding regular
decomposition classes~do.
\end{corollary}

\begin{proof}
One implication is the previous proposition, the converse follows from the
equivalence of~(1) and~(2) (\ref{pr:sp:dcdcent}) and the continuity of the
centralizer map~$\Centralizer$.  
\end{proof}

To put it another way, the centralizer map~$\Centralizer$ induces an
increasing isomorphism between the ordered set of regular decomposition
classes in~$\Gp$ and the ordered set of regular orbits in the variety of
reductions for a symmetric pair.

\subsubsection{Irregular locus}
The irregular locus~$\Irregular$ is the
complement~$\Gp\setminus\Regular$ of the set of regular elements. 
It was proven by~\VELDKAMP~\cite{VELDKAMP} for the symmetric pair associated
to the Cartesian square
of a group that the irregular locus~$\Irregular$ is the set of
points where the reductive quotient~$\phi:\Gp\to \Gp\goodquo K$ fails to be
submersive. However it remains unknown wether these equations span the ideal
of~$\Irregular$ or not. We noticed that these equations are geometrically
realized by the centralizer map~(\ref{pr:sp:irrcentralizer} below). This could be
helpful to determine if the ideal they span is reduced or not, or if the
caracterization of the irregular locus given by~\VELDKAMP\ extends to the
symmetric setting.

We define the Jacobian morphism~$J\phi$ associated
to the reductive quotient~$\phi:\Gp\to\Gp\goodquo K$ by
\[
\fun J\phi\: \Gp \times \Alt^r \Gp \to \C\\
(x,\xi)  \mapsto \Alt^r T_x \phi (\xi)
.
\\
\]
It is a homogeneous morphism, thus~$J\phi \in \Sym^N\Gp^* \otimes
\Alt^r\Gp^*$ for some~$N$. The restriction to~$\Gp$ of a non degenerate
bilinear form on~$\Gg$ that is invariant under~$G$ and~$\theta$ is
definite and~$K$-invariant. Hence we can see~$J\phi$ as
a~$K$-invariant rational map
\[
J\phi: \P(\Gp)\dashrightarrow \P(\Alt^r\Gp)
\]
by identifying~$\Gp$ with its dual.

\begin{proposition}
\label{pr:sp:irrcentralizer}
The Jacobian morphism~$J\phi$ coincides with
the centralizer map~$\Centralizer$.
\end{proposition}

\begin{proof}
Let~$\Ga$ be a \CARTAN\ subspace of~$\Gp$, $a\in\Ga^r$ a basis
of~$\Ga$ and~$x$ a regular element of~$\Ga$. We write~$\Gg'$ the
centralizer of~$\Ga$ in~$\Gg$ and~$\Gk' = \Gg' \cap \Gk$, $\Gp' = \Gg'
\cap \Gp'$, so that~$\ad x$ is a linear automorphism of~$\Gg'$
exchanging~$\Gk'$ and~$\Gp'$.

Now let~$\xi = \xi_1 \wedge\cdots\wedge\xi_r$ be a decomposed
$r$-vector divisible by a vector~$\xi_1$ belonging to~$\Gp'$. If~$f$
is a~$K$-invariant function and~$\epsilon\to0$, we have
\[
f(\exp(\epsilon \inv{(\ad x)}\xi_1) x ) = f(x) =
f(x) - \epsilon T_x f(\xi_1) + o(\epsilon)
.
\]
Hence~$T_x f(\xi_1) = 0$ and~$J\phi(x,\xi) = 0$ when~$\xi$ is divisible
by an element of~$\Gp'$. Now the reductive quotient is submersive at
the regular element~$x$~\cite[Theorem~13]{KR} so that~$J\phi(x,a_1 \wedge \cdots
\wedge a_r) \neq 0$. This shows that~$J\phi$ and~$\Centralizer$ agree
on~$\Regular$.
\end{proof}





\section{Rigidity of anisotropic tori}
\label{se:sp:ar}

\subsubsection{Review of stability}
Concepts familiar to the geometric invariant theory show up in the
study of degeneracies of semi-simple elements. Let us recall the
appropriate definitions and facts.
%
%
\begin{definition}
Let~$K$ be a reductive group, $V$ a finite dimensional representation
of~$V$ and~$X \subset \P(V)$ a quasi-projective variety stable under
the operation of~$K$. A point~$[v]$ of~$X$ is called
\begin{itemize}
\item
\emph{semi-stable} if there exists a $K$-invariant non-constant
homogeneous polynomial that does not vanish at~$v$;
\item 
\emph{poly-stable} if it is semi-stable and the~$K$-orbit of~$[v]$ is
closed in the set of semi-stable elements of~$X$;
\item 
\emph{stable} if it is poly-stable and has finite stabilizer in~$K$;
\item
\emph{unstable} if it is not semi-stable.
\end{itemize}
\end{definition}
When~$X$ is closed subvariety of~$\P(V)$, its semi-stable, poly-stable,
stable or unstable points are the points of~$\P(V)$ of the same kind,
belonging to~$X$. We say a non-zero vector~$v\in V$ is semi-stable,
poly-stable, stable or unstable, according to the nature of~$[v]$. The
following proposition is a basic result in geometric
invariant theory:

\begin{proposition}
Let~$X$ be a closed subvariety of~$\P(V)$ stable under the operation
of~$K$. Then:
\begin{enumerate}
\item 
The point~$[v]$ is semi-stable if, and only if, $0$ does not
belong to the closure of the~$K$-orbit of~$v$ in~$V$.
\item
The point~$[v]$ is poly-stable if, and only if, the $K$-orbit of~$v$
is closed in~$V$.
\end{enumerate}
\end{proposition}

According to~\cite{KR} these
propositions can be rephrased the following way in the setting of reductive
symmetric pairs:
\begin{itemize}
\item
the set of unstable elements in~$\Gp$ is precisely the set~$\Nilpotent$ of
nilpotent elements, with~$0$ removed;
\item
the set of poly-stable elements is the set~$\Semisimple$ of
semi-simple elements, with~$0$ removed;
\item
the set of semi-stable but not poly-stable elements is precisely the
set of elements with non-zero semi-simple part and non-zero nilpotent
part;
\item
the set of stable elements is not empty if, and only if, the rank of the
symmetric pair~$(G,\theta)$ equals the rank of reductive group~$G$; in this
case the set of stable elements is the set of anisotropic semi-simple
elements whose centralizer in~$\Gp$ is a \CARTAN\ subspace of~$\Gg$.
\end{itemize}

The following property of anisotropic algebras is remarkable with respect to
stability theory: for any anisotropic algebra~$\Ga$, the set~$\GaS$ of
poly-stable elements of~$\Ga$ (with zero added), and the set~$\GaN$ of
unstable elements of~$\Ga$ (\emph{idem}) are vector subspaces of~$\Ga$.

\subsubsection{Degeneracies of anisotropic algebras}
\label{ss:sp:dgcompl}
We introduce a language suited to our study of degeneracies.
Let~$K$ be a group acting on a complete variety~$X$. We consider a
point~$x_1 \in X$ and a \emph{degeneration}~$x_0$ of~$x_1$, that is, a point
belonging to the closure in~$X$ of the orbit~$K x_1$. There exists
an irreducible smooth curve~$C^o$ in~$K$ such that~$x_0$ belongs to the closure
of~$C^o x_1$, and we only consider degenerations along such curves.
In the study of the degeneration along the curve~$C^o$,
it is convenient to introduce the points at infinity of~$C^o$.

Let~$C$ be the smooth completion of~$C^o$. The points in~$C\setminus
C^o$ are called points at infinity. Since~$X$ is complete, the rational
map sending~$c\in C^o$ to~$c x_1$ extends to a regular map on~$C$ and,
by a slight abuse of notation, we write~$c x_1$ for the image of~$c$
under this map even for~$c$ lying in~$C\setminus C^0$. Now let~$c_0$ be a point
at infinity whose image is~$x_0$. We say that~$C^o$ is an \emph{arc in~$K$
pushing~$x_1$ toward~$x_0$}, and~$c_0$ is a \emph{point at infinity
taking~$x_1$ to~$x_0$}. 

\medbreak
The way semi-simple elements of an anisotropic algebra behave when it is
deformed by the operation of~$K$ is described by the following theorem. It
has many interesting consequences, such as~\ref{co:sp:ar},
\ref{pr:sp:srdescr} and~\ref{th:sp:dcclosure}.

\begin{theorem}
\label{th:sp:ardg}
Let~$\Ga_1$ be an~$l$-dimensional anisotropic subalgebra
and~$\Ga_0\in\Grass(l,\Gp)$ lying in the closure of the~$K$-orbit through~$\Ga_1$. 
Let~$C^o$ be an arc in~$K$ pushing~$\Ga_1$
toward~$\Ga_0$ and~$c_0$ a point at infinity taking~$\Ga_1$
to~$\Ga_0$.

Then, unless~$\Ga_0$ only contains nilpotent elements, there exists a
semi-stable element~$s_0$ in~$\Ga_0$ and a semi-stable element~$s_1$
in~$\Ga_1$ pushed toward~$s_0$ by~$C^o$ and brought to~$s_0$ by~$c_0$.

Moreover, if~$\Ga_1$ is closed under the~\CJ\ decomposition, then~$\Ga_0$ is
closed under the~\CJ\ decomposition as well, and for each semi-simple
element~$s_0$ in~$\Ga_0$ there exists a semi-simple element~$s_1$ in~$\Ga_1$
that is brought to~$s_0$ by~$c_0$.
\end{theorem}

\begin{proof}
We denote by~$R_0$ the local ring of~$C^o$ at~$c_0$, $\Gm_0$ its maximal ideal
and~$\epsilon$ a local parameter of~$C^o$ at~$c_0$. According
to~(\ref{th:sp:mfbasis}) there exists a basis~$x$ of~$\Ga_1$, a
basis~$y$
of~$\Ga_0$, a $r$-uple $n$~of integers and a~$r$-uple~$\eta\in(\Gm_0\otimes
\Gp)^r$ such that for all~$k$
\[
cx_k = \epsilon^{n_k}(y_k + \eta_k)
.
\]
If all of the~$y_k$ are nilpotent, then~$\Ga_0$ purely consists of nilpotent
elements. Hence, we assume that for some integer~$k$ the vector~$y_k$ is
semi-stable and let~$f$ be a $K$-invariant function, homogeneous of
degree~$N > 0$, that does not vanish at~$y_k$. We compute
\[
f(x_k) = \epsilon^{N n_k} f(y_k + \eta_k)
,
\]
and show that~$n_k = 0$. On the one hand, the function~$f(x_k)/f(y_k +
\eta_k)$ is regular at~$c_0$, which implies~$n_k \ge 0$. On the other hand,
the non zero function~$f(x_k)\epsilon^{-Nn_k}$ is regular at~$c_0$, which
implies~$n_k \le 0$. Thus~$n_k$ equals zero and~$f$ does not vanish at~$x_k$. We
conclude
that~$cx_k$ converges to~$y_k$ when~$c$ approaches~$c_0$ and~$x_k$ is
semi-stable.

We now proceed to the proof of the second part of the statement and assume
that~$\Ga_1$ is closed under the~\CJ~decomposition. According
to~(\ref{pr:sp:mfdeco}) we may assume that each of
the~$x_k$ is either semi-simple or nilpotent. We show that each of
the~$y_k$ is either semi-simple or nilpotent.

The image of~$\Nilpotent$ in~$\P(\Gp)$ is closed, hence~$x_k\in\Nilpotent$
implies~$y_k\in\Nilpotent$.
We now assume that~$x_k$ is semi-simple. If~$y_k$ is not
nilpotent, then~$n_k = 0$ as before and we have~$
x_k = y_k + \eta_k
$.
Hence~$c x_k$ converges toward~$y_k$ as~$c$
approaches~$c_0$. But~$x_k$ is semi-simple and its orbit is closed, so~$y_k$
is semi-simple and conjugated to~$x_k$. Thus each~$y_k$ is either
semi-simple or nilpotent, which implies that the
anisotropic subalgebra~$\Ga_0$  of dimension~$l$ spanned by the~$y_k$ is closed
under the~\CJ\ decomposition.

Last, if~$s_0\in\Ga_0$ is semi-simple, it is a linear combination~$
s_0 = \sum_{k\in K_s} Y_0^k y_k$ of the semi-simple~$y_k$'s, whose set of indice
is~$K_s \subset\set{1,\dots,l}$. Hence the translates~$cs_1$ of~$s_1 =
\sum_{k\in K_s} Y_0^k x_k$ converge toward~$s_0$ as~$c$ approaches~$c_0$.
\end{proof}

If a~$l$-dimensional anisotropic subalgebra~$\Ga_1$ of~$\Gg$ degenerates
on~$\Ga_0$, the semi-simple elements subsisting in~$\Ga_0$ can
be~\emph{rigidified}, that is, there exists an equivalent degeneration leaving
these semi-simple elements untouched:

\begin{corollary}[Rigidity of anisotropic tori]
\label{co:sp:ar}
Let~$\Ga_1$ be a~$l$-dimensional anisotropic
subalgebra closed under~\CJ\ decomposition,~$\Ga_0\in\Grass(l,\Gp)$ lying in
the closure of the~$K$-orbit
through~$\Ga_1$ and~$C^o$ be an arc in~$K$ pushing~$\Ga_1$
toward~$\Ga_0$ and~$c_0$ a point at infinity taking~$\Ga_1$
to~$\Ga_0$. Let~$\SemisimplePart{\Ga_0}$ be the linear subspace of~$\Ga_0$
spanned by its semi-simple elements. Then there exists a
$K$-conjugate~$\Ga'_1$ of~$\Ga_1$ containing~$\SemisimplePart{\Ga_0}$ and a
degeneracy curve in the centralizer~$C_K(\SemisimplePart{\Ga_0})$
of~$\SemisimplePart{\Ga_0}$ in~$K$ pushing~$\Ga'_1$
toward~$\Ga_0$.
\end{corollary}

This corollary is a consequence of~\ref{th:sp:ardg} and the
following general lemma dealing with degeneracies of orbits:

\begin{lemma}[Straightening of degeneracies]
\label{le:sp:dgstraight}
Let~$K$ be a connected group and~$X$ and~$Y$ be
two~$K$-varieties. Let~$x_1$ be a point in~$X$ and~$x_0$ a point
belonging to the closure of the~$K$-orbit through~$x_1$, and let~$y_1$ a
point in~$Y$. Let~$C^o$ be a degeneracy curve in~$K$ pushing~$x_1$
toward~$x_0$ and~$c_0$ a point at infinity taking~$x_1$ to~$x_0$. If~$c_0$
takes~$y_1$ to a conjugate~$y_0 = hy_1$, then there exists a degeneracy curve
in the identity component of the stabilizer of~$y_0$ in~$K$
pushing~$hx_1$ to~$x_0$, and a point at infinity taking~$hx_1$ to~$x_0$.
\end{lemma}

\begin{proof}
The rational map sending~$c$ to~$cy_1$ is regular at~$c_0$ and~$c_0
y_1 $ is a conjugate~$y_0 = h y_1$ of~$y_1$. Let~$C(c_0) = C^o \cup
\set{c_0}$ and define
\[
Z = \set{ (g,c) \in K\times C(c_0) \mid gcy_1 = y_0 }
.
\]
The map~$cy_1$ is regular at~$c_0$, and its value belongs to the orbit
of~$y_1$, thus~$(1,c_0)\in Z$ belongs to the closure of the inverse
image~$W$ of~$C^o$ under the projection of~$Z$ to~$C(c_0)$. Let~$D$
be a degeneracy curve in~$W$ containing~$(1,c_0)$ in its closure. Then the
image~$(C^o)'$ of~$W$ under the map sending~$(g,c) \in Z$ to~$gc\inv
h\in K$ is contained in the centralizer of~$y_0 \in K$ and
pushes~$hx_1$ toward~$x_0$. This map extends to a regular map
at~$(1,c_0)$ whose image is a point at infinity for~$(C^o)'$
taking~$hx_1$ to~$x_0$.
\end{proof}





\section{Algebraicity of reductions}
\label{se:sp:alg}

\subsubsection{Degeneracies of subgroups}

The fact that each point in the variety of reductions is the \LIE\ algebra of
an algebraic subgroup of~$G$ is a consequence of the following general
observation:

\begin{proposition}
\label{pr:sp:algdg}
Let~$G$ be an algebraic group and~$\mathcal{H} \subset G \times B $ an
irreducible family of subvarieties of~$G$ over a basis~$B$. If the generic member
of~$\mathcal{H}$ is a subgroup of~$G$, then each member of~$\mathcal{H}$ is a
subgroup of~$G$.
\end{proposition}

\begin{proof}
The algebraic subgroups of~$G$ are precisely the subvarieties~$H$ of~$G$
such that the restriction of the morphism~$\psi:G\times G \to G$
defined by $\psi(g_1,g_2) = g_1 g_2^{-1}$ maps~$H\times H$
in~$H$. Let
\[
B^o = \set{b \in B \mid \textrm{$\mathcal{H}_{b}$ is a subgroup of~$G$}}
\]
and~$\mathcal{H}^o$ the restriction of~$\mathcal{H}$ to~$B^o$, and assume
that~$B^o$ is dense in~$B$. Since~$\mathcal{H}$ is
irreducible,~$\mathcal{H}^o$ is dense
therein. Hence~$\inv\psi(\mathcal{H})$
contains~$\mathcal{H}\times\mathcal{H}$, the closure
of~$\mathcal{H}^o\times\mathcal{H}^o$.
\end{proof}

\begin{corollary}
\label{co:sp:algred}
Any point of the variety of reductions is the \LIE\ algebra of an
algebraic subgroup of~$G$.
\end{corollary}

\begin{proof}
Let~$\Ga_0$ be a point of~$\Red$, $\Ga_1$ a \CARTAN\ subspace of~$\Gp$
and~$C^o$ a degeneracy curve in~$K$ pushing~$\Ga_1$ toward~$\Ga_0$. We denote
by~$A_1$ be the connected subgroup of~$G$ whose \LIE\ algebra is~$\Ga_1$,
by~$\bar G$ a projective completion of~$G$ and by~$C$ a smooth completion
of~$C^o$. We consider the family~$\mathcal{V}\subset \bar G \times \P(\Gp)
\times C$ obtained by taking the closure of
\[
\mathcal{V}^o = \set{
(g,[x],c) \in G \times \P(\Gp) \times C^o
\mid
\textrm{$g \in c A_1$ and $x \in c\Ga_1$}
}
.
\]
Let~$c_0$ be a point of~$C$ taking~$\Ga_1$ to~$\Ga_0$. The
fiber~$\mathcal{V}_{c_0}$ projects on~$\bar G$ as a set meeting~$G$ along a
subgroup~$A_0$ of~$G$~(\ref{pr:sp:algdg}), and on~$\P(\Gp)$ as the
projectivization of~$\Ga_0$. Hence the \LIE\ algebra of~$A_0$
contains~$\Ga_0$. But the family~$\mathcal{V}$ over~$C$ is flat, since the
curve~$C$ is smooth~\cite[9.7]{HART}, and the dimensions of~$A_1$
and~$\Ga_1$ agree~\cite[9.5]{HART}.
\end{proof}

\subsubsection{Nilpotence of closed orbits}


Degeneracies in the varieties of reductions ultimately turn each semi-simple
dimension into a nilpotent one, except of course those belonging to the center
of~$\Gg$:

\begin{proposition}
\label{pr:sp:algclosed}
If~$G$ is semi-simple, then a point in a closed orbit of~$\Red$ is contained
in the nilpotent cone~$\Nilpotent$ of~$\Gg$.
\end{proposition}

\begin{proof}
Let~$\Gu$ be a reduction belonging to a closed~$K$-orbit in~$\Red$, and~$P$ its
stabilizer in~$K$. The \LIE\ algebra~$\Gu$ is
algebraic~(\ref{co:sp:algred}) and abelian, it follows from
the \CJ-decomposition that the set~$\SemisimplePart{\Gu}$ of semi-simple elements
of~$\Gu$ and the set~$\NilpotentPart{\Gu}$ of nilpotent elements of~$\Gu$ are
supplementary linear subspaces of~$\Gu$. The stabilizer in~$K$
of~$\SemisimplePart{\Gu}$ hence contains~$P$. It is thus both a reductive and
a parabolic subgroup of~$K$, so it equals~$K$. This yields~$\SemisimplePart{\Gu}
= 0$.
\end{proof}





\section{Subvarieties of reductions}
\label{se:sp:sr}

\subsubsection{Variety of reductions and derivation}
To the reductive symmetric pair~$(G,\theta)$ corresponds a symmetric
pair whose group is the derived group~$G'$ of~$G$. Since~$G'$ is
stabilized by any automorphism of~$G$, the involution~$\theta$ restricts
to~$G'$, turning~$(G',\theta)$ into a semi-simple symmetric pair.

\begin{proposition}
\label{pr:sp:srcenter}
Let~$(G,\theta)$ be a symmetric pair and~$G'$ the derived group
of~$G$. Let~$K'$ be the adjoint form of the identity component of the
fixed points group of~$(G',\theta)$. Let~$\Red$ be the variety of reductions
for~$(G,\theta)$ and~$\Red'$ the one for~$(G',\theta)$. Let~$\Gp'$ be the
space of anisotropic vectors of~$(G,\theta)$ belonging to the \LIE\ algebra
of~$G'$, and~$\Gp_Z$ those belonging to the \LIE\ algebra of the
center of~$G$. Then:
\begin{enumerate}
\item
The groups~$K$ and~$K'$ have the same orbits in~$\Red$.
\item
The maps~$p:\Red\to\Red'$ and~$j:\Red'\to\Red$ defined by
\[
  p(\Gu) = \Gu\cap\Gp'
  \qquad\textrm{and}\qquad
  j(\Gv) = \Gv\oplus\Gp_Z
\]
are inverse $K'$-invariant isomorphisms.
\end{enumerate}
\end{proposition}

\begin{proof}
To establish the first claim, we replace~$G$ with its adjoint form so
that~$G$ is the direct product of its derived group~$G'$ and its
center~$Z_G$. Now~$G^\theta = (G')^\theta \times Z_G^\theta$, and
since~$Z_G^\theta$ lies in the kernel of the map~$G^\theta\to
\Aut(\Red)$, the map~$K\to\Aut(\Red)$ factors through~$K'$.

To prove the second claim, it is sufficient to show that for
any~$\Gu\in\Red$,
\begin{equation}
\label{eq:sp:srcenter}
  \Gu = (\Gu \cap \Gp') \oplus \Gp_Z
\end{equation}
holds.
We notice that any reduction contains~$\Gp_Z$. The
infinitesimal anisotropic center~$\Gp_Z$ is pointwise fixed by~$K$,
thus it is is enough to remark that any ordinary reduction
contains~it, which follows from the maximality condition.
Now~$\Gp = \Gp' \oplus \Gp_Z$ and~(\ref{eq:sp:srcenter}) holds.
\end{proof}

We may narrow our study of varieties of reductions for reductive
symmetric pairs down to semi-simple symmetric pairs. However,
reductive symmetric pairs appear naturally as symmetric pairs
associated with centralizers of anisotropic tori.

\subsubsection{Subvarieties of reductions}

Let~$\Ga\in\Red_o$ an ordinary reduction and~$\tilde\Ga \subset\Ga$ a subspace
of~$\Ga$. The centralizer~$C_G(\tilde\Ga)$ is a~$\theta$-stable reductive
subgroup of~$G$, thus~$(C_G(\tilde\Ga),\theta)$ is a reductive symmetric
pair whose rank equals the rank of~$(G,\theta)$.
Since~$\gcentralizer(\tilde\Ga) \subset \Gg$, the
variety~$\Red(\tilde\Ga)$ of reductions for the reductive symmetric
pair~$(C_G(\tilde\Ga),\theta)$ is a natural subvariety of the variety~$\Red$
of reductions for the reductive symmetric pair~$(G,\theta)$. It is naturally
isomorphic to the variety of reductions~$(C'_G(\tilde\Ga),\theta)$ whose group
is the derived group of~$C_G(\tilde\Ga)$.

\begin{proposition}
\label{pr:sp:srdescr}
Let~$\tilde\Ga$ be a subspace of an ordinary reduction and~$\Red(\tilde\Ga)
\subset \Red$ the variety of reductions for the symmetric
pair~$(C_G(\tilde\Ga), \theta)$. Then
\[
\Red(\tilde\Ga) = \set{ \Gu\in\Red \mid \Gu \supset \tilde\Ga }
.
\]
\end{proposition}

\begin{proof}
The variety of reductions~$\Red(\tilde\Ga)$ is~\emph{a priori} a subset of the
right hand side. The reciprocal inclusion follows from ~(\ref{co:sp:ar}).
\end{proof}

\begin{definition}
The \emph{reductive symmetric pair centralizing~$\tilde\Ga$} is the reductive
symmetric pair~$(C_G(\tilde\Ga),\theta)$. The \emph{variety of reductions
containing~$\tilde\Ga$} is the variety of reductions~$\Red(\tilde\Ga)$
of~$(C_G(\tilde\Ga),\theta)$. The subvarieties of reductions of~$\Red$ are the
subvarieties of the form~$\Red(\tilde\Ga)$.
We write~$K(\tilde\Ga)$ the adjoint form of the identity component of the
fixed points group~$C_G(\tilde\Ga)^\theta$. Recall that this is the group whose
action on~$\Red(\tilde\Ga)$ we are interested in.
\end{definition}

Note that 
the space~$\tilde\Ga$ is tangent to the center of~$C_G(\tilde\Ga)$. Hence
the semi-simple symmetric pair associated to~$(C_G(\tilde\Ga),
\theta)$~(\ref{pr:sp:srcenter}) has smaller rank than~$(G,\theta)$. The
following propositions are immediate consequences of the rigidity
of~$\tilde\Ga$~(\ref{co:sp:ar}).

\begin{proposition}
Let~$\Red(\tilde\Ga)$ be a subvariety of reductions of~$\Red$ and~$\Orbit$ an
orbit of~$K$ in~$\Red$, whose codimension is~$k$. Then~$\Orbit \cap
\Red(\tilde\Ga)$ is an orbit
of~$C_K(\tilde\Ga)$ in~$\Red(\tilde\Ga)$ whose codimension is~$k$. In
particular~$\Orbit \cap \Red(\tilde\Ga)$ is a finite
union of orbits of~$K(\tilde\Ga)$.
\end{proposition}

\begin{proposition}
Let~$\Red(\tilde\Ga)$ be a subvariety of reductions of~$\Red$ and~$\Ga_0$ a
point of~$\Red(\tilde\Ga)$.
If~$\Orbit$ is a~$K$-orbit in~$\Red$ containing~$\Ga_0$
in its closure, then~$\Orbit$ contains an orbit of~$K(\tilde\Ga)$
in~$\Red(\tilde\Ga)$ containing~$\Ga_0$ in its closure.
\end{proposition}

\begin{proposition}
Let~$\Ga_1$ and~$\Ga_0$ be two reductions, and let us assume that~$\Ga_0$ is
not contained in the nilpotent cone~$\Nilpotent$ of~$\Gp$. Then, the following
statements are equivalent.
\begin{enumerate}
\item
The $K$-orbit in~$\Red$ containing~$\Ga_1$ is more general than the one
containing~$\Ga_0$.
\item
For any subvariety of reductions~$\Red(\tilde\Ga)$ containing~$\Ga_0$, there
exists a~$K$-conjugate of~$\Ga_1$ in~$\Red(\tilde\Ga)$
whose~$K(\tilde\Ga)$-orbit is more general than the one of~$\Ga_0$.
\item
There exists a subvariety of reductions~$\Red(\tilde\Ga)$ containing~$\Ga_0$
and a~$K$-conjugate of~$\Ga_1$ in~$\Red(\tilde\Ga)$
whose~$K(\tilde\Ga)$-orbit is more general than the one of~$\Ga_0$.
\end{enumerate}
\end{proposition}





\section{Special reductions}
\label{se:sp:reds}

The set~$\Red_s = \Red\setminus\Red_o$ of special reductions 
is the complement of an affine open set in a
projective variety, hence it is a
divisor in~$\Red$. This divisor is hard to describe in the general setting: we
can not even count its irreducible components. We are however able to show
that it is cut out by a smooth quadric of the ambiant space~(\ref{se:sp:reds}).

\subsubsection{Bilinear algebra and exterior algebra}
A~symmetric bilinear form~$b$ on a finite dimensional vector space~$E$ induces
a symmetric bilinear form~$\Alt^r b$ on the~$r$-th exterior power of~$E$. On~two decomposed $r$-vectors $u_1 \wedge \cdots \wedge u_r$ and~$v_1
\wedge \cdots \wedge v_r$ this bilinear form evaluates to the determinant of
the matrix with coefficients~$b(u_i,v_j)$. When~$b$ is regular,
so is~$\Alt^r b$.

\subsubsection{Special reductions}
The \KILLING\ form~$b$ of the reductive \LIE~algebra~$\Gg$ is preserved by
automorphisms of~$\Gg$. Consequently, the characteristic spaces~$\Gk$ and~$\Gp$
of~$\theta'$ are orthogonal, since they are also supplementary, the restriction
of~$b$ to any of them is regular. The quadratic form~$q$ associated to~$\Alt^r
b$ on~$\P(\Alt^r \Gp)$ defines a smooth quadric~$Q$.

\begin{proposition}
\label{pr:sp:redsing}
The set~$\Red_s$ of special reductions is the intersection of the variety of
reductions~$\Red$ with the quadric~$Q$.
\end{proposition}

To begin with, we state a proposition and two lemmas.

\begin{proposition}[\refcite{37.5.2}{TY}]
The bilinear form induced by the \KILLING\ form on a \CARTAN\ subspace
of~$\Gp$ is regular.
\end{proposition}

\begin{lemma}
The bilinear form induced by the \KILLING\ form on an anisotropic subalgebra
of~$\Gp$ containing a nilpotent element is degenerate.
\end{lemma}

\begin{proof}
Let~$\Gu$ be an anisotropic subalgebra of~$\Gp$ containing a nilpotent
element~$n$. Since~$\Gu$ is abelian, endomorphisms~$\ad u \circ \ad n$ are
nilpotent for each~$u\in\Gu$, so that~$b(u,n) = 0$. Hence~$n$ lies in the
kernel of the restriction of~$b$ to~$\Gu$.
\end{proof}

\begin{lemma}
An anisotropic algebra contains a nilpotent element, unless it is an ordinary
reduction.
\end{lemma}

\begin{proof}[Proof of the lemma]
Let~$\Gu\in\Ab$ be an anisotropic subalgebra of~$\Gg$ of
dimension~$r$ and let~$\Ga$ be a \CARTAN\ subspace of~$\Gg$ containing
the semi-simple parts of the elements of~$\Gu$. The~\CJ-decomposition induces
a linear map~$\Gu \to \Ga$ whose kernel is the set of nilpotent elements
in~$\Gu$. Since~$\pcentralizer(\Ga) = \Ga$~(\ref{pr:sp:spcentralizer})
the image of this linear map has rank~$r$ only if~$\Gu = \Ga$.
\end{proof}

\begin{proof}[Proof of~\ref{pr:sp:redsing}]
It follows from the previous lemmas that the set of reductions~$\Gu$ for which
the restriction of the \KILLING\ form to~$\Gu$ is degenerate is exactly the
set of special reductions. These points are also the ones where~$q$ vanishes,
so that~$\Red_s = \Red\cap Q$.
\end{proof}





\section{Partial positivity of the anticanonical class}
\label{se:sp:pp}

\subsubsection{Linear subspaces of \GRASSMANN\ varieties}
Recall that linear subspaces of $\ProjGp$ contained in~$\GrassGp$
are precisely the \GRASSMANN\ subvarieties of~$\GrassGp$ which are
also projective spaces. \GRASSMANN\ subvarieties of~$\GrassGp$ are the
sets
\[
  \Gamma(\Gv, \Gw) = \set{
    \Gu\in\Grass(r,\Gp)\mid \Gv\subset\Gu\subset\Gw
  }
\]
where~$\Gv$ and~$\Gw$ are linear subspaces of~$\Gp$. Its dimension is
$\dim{(\Gu/\Gv)}\dim{(\Gw/\Gu)}$, and it
is a linear space precisely when~$\dim{(\Gu/\Gv)} = 1$
or~when~$\dim{(\Gw/\Gu)} = 1$.

\subsubsection{Maximal linear subspaces of varieties of reductions}
Let us recall that a singular torus of~$S$ is a torus whose centralizer in~$G$
is not a maximal torus. We say an anisotropic reductive algebra
is \emph{singular} when its centralizer in~$\Gp$ is not a \CARTAN\ subspace.
Maximal singular anisotropic reductive algebras are the kernels of the roots
of~$\Gg$ relative to some \CARTAN\ subspace~$\Ga$~(\ref{eq:sp:spweight}).

\begin{proposition}
Let~$\Gz$ be a maximal singular anisotropic reductive algebra. The
subspace~$\Gamma(\Gz) = \Gamma(\Gz, \pcentralizer(\Gz))$ of~$\GrassGp$ is 
contained in~$\Red$, it pass through all points in~$\Red_o$
containing~$\Gz$.
\end{proposition}

\begin{proof}
Let~$\Ga$ be a \CARTAN\ subspace containing~$\Gz$ and~$\alpha$ a root of~$\Gg$
relative to~$\Ga$ whose kernel is~$\Gz$. The open subset~$\Red_o$ of the
irreducible component~$\Red$ of~$\Ab$ meets at~$\Ga$ the linear
space~$\Gamma(\Gz)$ contained in~$\Ab$. This linear space is therefore
contained in~$\Red$.
\end{proof}

\begin{theorem}
The linear subspaces~$\Gamma(\Gz)$ are maximal among the linear subspaces
of~$\Red$ passing through a general point.
\end{theorem}

\begin{proof}
Let~$\Ga\in\Red_o$ be a general point of~$\Red$, and~$\Gamma = \Gamma(\Gv,\Gw)$ a
linear subspace of~$\Red$ passing through~$\Ga$. We shall see
that~$\Gv$ needs to be a maximal singular anisotropic reductive algebra in
order to let~$\Gamma$ be maximal.

Assume that~$\Gv$ has codimension greater than~two in~$\Ga$, each pair
in~$\Ga\times\Gv$ has its members belonging to a point of~$\Gamma$,
now~$\Gamma\subset\Ab$,
hence~$\Gv\subset\pcentralizer(\Ga)$. But~$\pcentralizer(\Ga) =
\Ga$~(\ref{pr:sp:spcentralizer}) and~$\Gamma = \set{\Ga}$ is not
maximal.

Assume now that~$\Gv$ has codimension~one in~$\Ga$. Since~$\Gamma\subset\Ab$,
we have~$\Gv\subset\pcentralizer(\Gw)$. But this centralizer is
\[
\Ga \oplus \bigoplus \Gp(\Ga)_\alpha
\]
where the sum extends over the set of positive roots~$\alpha$ relative
to~$\Ga$ which vanish on~$\Gv$. In order to let~$\Gw$ be strictly bigger
than~$\Ga$, this set of roots must not be empty. There also exists a
root~$\alpha$ whose kernel contains~$\Gv$, and for dimension reasons~$\Gv =
\ker\alpha$ is a maximal singular anisotropic reductive
subalgebra. Thus~$\Gamma\subset\Gamma(\Gv)$, but~$\Gamma$ is maximal, it
equals~$\Gamma(\Gv)$. In case~$\pcentralizer(\Gv) = \Ga$ for each codimension
one subspace~$\Gv$ in~$\Ga$, the root system of~$\Gp$ relative to~$\Ga$ is
empty and~$\Gp=\Ga$: $\Red$ is a point and~$\Gamma(\Gv)$ is maximal.
\end{proof}

\begin{corollary}
Through a general point~$\Ga$ of~$\Red$ passes a finite number of maximal
linear subspaces of~$\Red$, meeting transversally in~$\Ga$. The intersection
of any two of these subspaces is~$\Ga$.
\end{corollary}

This follows from the decomposition~(\ref{eq:sp:spdecoaniso}). Notice that
these linear subspaces do not need to share a common dimension. However, we
can make this picture more accurate in the case of the variety of
\CARTAN\ reductions for a reductive group.

\begin{corollary}
Let~$\Red$ be the variety of \CARTAN\ reductions for a reductive group of
rank~$r$ and dimension~$2m + r$. Through a general point of~$\Red$ passes~$m$
projective planes, any two of them meeting transversally in this point.
\end{corollary}

\begin{proposition}
\label{pr:sp:linesmoothlocus}
A general line contained in~$\Red$ is contained in the smooth locus of~$\Red$.
\end{proposition}

\begin{proof}
We show that a general line is contained in the image of the centralizer
map~(\ref{th:sp:redregular}). Such a line is contained in a
space~$\Gamma(\Gz)$, hence we can replace~$\Red$ by the variety of
reductions for the pair~$(C'_G(\Gz), \theta)$, whose group is the derived
group of the centralizer of~$\Gz$
in~$G$~(see~\ref{pr:sp:srdescr},
\ref{pr:sp:srcenter}). We are then reduced to the case of a reductive
symmetric pair of rank~one, where~$\Red = \P(\Gp)$. The generic nilpotent
element is regular~\cite[Theorem~3]{KR}, so that the irregular locus has
codimension at least~$2$ in~$\Red = \P(\Gp)$: a~generic line will miss~it.
\end{proof}

\subsubsection{Canonical class}

\begin{proposition}
\label{pr:sp:pphsmooth}
Let~$X$ be a smooth algebraic variety quasi homogeneous under the action of
an algebraic group~$G$. For any smooth rational curve~$C$ in~$X$ touching the open
orbit of~$X$, the \HILBERT\ scheme~$\HilbertS$ parametrising deformations of~$C$
in~$X$ is smooth at~$[C]$.
\end{proposition}

\begin{proof}
We show that the second cohomology group~$H^1(N_{C/X})$ of the normal bundle
to~$C$ in~$X$ vanishes, which occurs only if the \HILBERT\ scheme is
smooth at~$[C]$~\cite[Theorem~2.6]{DEBARRE}.

Let~$\theta:\Gg\times X \to TX$ be the morphism obtained by restricting the
map tangent to the group action~$G \times X \to X$. (Recall that $X$ is
embedded in~$TX$ \emph{via} the zero section.) Partial application of~$\theta$
gives a global section~$\theta_a: X \to TX$ of the tangent bundle from
any~$a\in\Gg$, and at any point~$x$ in the open orbit~$X_o$ of~$X$ under~$G$, the
stalk~$(TX)_x$ of the tangent bundle is generated by these global
sections. Thus, the restriction of the tangent bundle of~$X$ to~$C_o = X_o
\cap C$ is globally generated and~$\left.N_{C/X}\right\vert_{C_o}$ is
generated by global sections of~$N_{C/X}$.

Since~$C$ and~$X$ are smooth, the normal sheaf to~$C$ in~$X$ is locally free.
It splits as
\[
  N_{C/X} = \Struct(a_1)\oplus\cdots\oplus\Struct(a_r)
\]
for some integer vector~$a\in\Z^r$, and~$\Struct(1)$ being the tautological
bundle on~$C\simeq\P^1$. If~$a_i$ is negative, the bundle~$\Struct(a_i)$
has no global sections, and at any point~$c\in C_o$, global sections span at
most a hyperplane in~$\left.N_{C/X}\right\vert_c$. Therefore the
the~$a_i$'s are non negative and thus~$H^1(N_{C/X})= 0$.
\end{proof}

\begin{corollary}
\label{co:sp:pphsmooth}
Let $\Ga\in\Red_o$ be a general point of~$\Red$ and~$\Delta$ a line through~$\Ga$
contained in~$\Red$. Let~$m$ be the dimension of the maximal linear subspace
of~$\Red$ through~$\Ga$ containing~$\Delta$. If~$\Delta$ is contained in the
smooth locus of~$\Red$, then
\[
  - K_r \cdot \Delta = m + 1
\]
where~$K_r$ is the canonical class of the smooth locus of~$\Red$.
\end{corollary}

\remark
The hypothesis on~$\Delta$ is always satisfied when this line is
sufficiently general~(\ref{pr:sp:linesmoothlocus}). When the variety of
\CARTAN\ reductions for a reductive group is under
consideration, we have~$- K_r\cdot\Delta = 3$, for any such line.

\begin{proof}
\newcommand\NormalBundle{{N_{\Delta/\Red}}}
According to the \RR\ formula,
\[
  \chi(\NormalBundle) =
    c_1(\NormalBundle)
    + \rank(\NormalBundle)(1 - g(\Delta))
    ,
\]
the rank~$\rank(\NormalBundle)$ is~$\dim\Red - \dim\Delta$ and the
genus~$g(\Delta)$ of~$\Delta\simeq\P^1$ vanishes. We compute:
\begin{align*}
  c_1(N_{\Delta/\Red}) 
  &= c_1(\left.T\Red\right\vert_\Delta) - c_1(T\Delta)\\
  &= c_1(\det\left.T\Red\right\vert_\Delta) -c_1(\Reg(2))\\
  &= - K_r \cdot \Delta - 2
  .
\end{align*}
By~(\ref{pr:sp:pphsmooth}) the~\HILBERT\ scheme~$\HilbertS_{[\Delta]}$
parametrising deformations of~$\Delta$ in~$\Red$ is smooth at~$[\Delta]$. Its
dimension at~$[\Delta]$ is~$h^0(\NormalBundle)$ and~$h^1(\NormalBundle) = 0$.
The~\RR\ formula eventually yields
\[
  h^0(\NormalBundle) = - K_r \cdot\Delta + \dim\Red - 3
  .
\]
We compute the dimension~$h^0(\NormalBundle)$ of~$\HilbertS_{\Delta]}$
another way. Put
\[
  Z = \set{ 
    (\Ga, \Alpha) \in\Red\times\HilbertS_{[\Delta]}
    \mid
    \Ga\in\Alpha
  }
  .
\]
The~$Z$ component~$(\Ga,\Delta)$ belongs to, projects to a subset of~$\Red$
containing~$\Red_o$ with general fiber of dimension~$m-1$, while this same
component projects to~$\HilbertS_{[\Delta]}$ with general fiber of
dimension~$1$. We can then compute
\[
  \dim\Red + m - 1 = 1 + h^0(\NormalBundle)
\]
hence~$K_r \cdot \Delta = -m - 1$.
\end{proof}





\section{Variety of anisotropic algebras}
\label{se:sp:aa}

In this section, the symmetric pair associated to the Cartesian square of a
simple group~$G$ of rank~$r$ is under consideration.

\subsubsection{Rough estimate of the dimension of a nilpotent orbit}
We say an orbit~$\Orbit$ in~$\Ab$ is~\emph{nilpotent} if it is the orbit of an
abelian algebra contained in the nilpotent cone~$\Nilpotent$ of~$\Gg$.

\begin{proposition}
\label{pr:sp:aanilporbit}
The dimension of a nilpotent orbit in~$\Ab$ is less than~$\dim G - r - 2$,
unless~$G$ has type~$A_1$.
\end{proposition}

\begin{proof}
\newcommand\vnormalizer{{\germ{n}_{\Gu}(\Gv)}}
Let~$B$ be a \BOREL\ subgroup of~$G$, $\Gu$ the \LIE\ algebra of its unipotent
radical and~$\Gv$ an abelian subalgebra of~$\Gu$ with dimension~$r$---any
nilpotent orbit in~$\Ab$ contains such a~$\Gv$. We bound from below the dimension
of the normalizer of~$\Gv$ in~$\Gg$ by the dimension of the
normalizer~$\vnormalizer$ of~$\Gv$ in~$\Gu$, thus obtaining a bound from above
for the dimension of the orbit through~$\Gv$.

According to the \LIE-\KOLCHIN\ theorem, endomorphisms of~$\Gu/\Gv$ adjoint to
elements in~$\Gv$ share a common nilvector, so~$\dim{\vnormalizer/\Gv} \ge
1$. Assume that this dimension is~$1$, and let~$v$ be the generic element
of~$\Gv$, and~$v'$ the endomorphism of~$\Gu/\Gv$ adjoint to~$v$. Since~$v$ is
generic, $\ker v' = \vnormalizer/\Gv$ has dimension~$1$ and~$v'$ is a cyclic
nilpotent endomorphism, whose~\JORDAN\ normal form consists of a single
block. The dimension of~$\vnormalizer/\Gv$ thus equals the degree of the minimal
polynomial of~$v'$. The endomorphism of~$\Gu$ adjoint to the generic element
of~$\Gu$ has minimal polynomial of degree~$h-1$, where~$h$ is
the~\COXETER\ number of~$G$. Since this minimal polynomial also vanishes
at~$v'$, we have
\[
  \dim \vnormalizer/\Gv \le h - 1
  .
\]
Let~$n/2 = \dim\Gu$ be the number of positive roots of~$G$, the former
inequality gives~$n/2 \le h + r - 1$ and from the
table~\ref{ta:sp:aacoxmal}~(see~\BOURBAKI~\cite{BBKI}) we infer this inequality
is only possible when~$G$ has type~$A_1$, $A_2$, $A_3$, $B_2$
or~$G_2$. Type~$A_1$ is not to be considered, types~$A_2$ and~$A_3$ have their
variety of \CARTAN\ reductions studied by~\IM~\cite{ILR,ILRGLN}, and
the remaining ones are studied in~\cite{COMPANION}.
\end{proof}

\begin{table}
\caption{Number of positive roots and \COXETER\ numbers}
\label{ta:sp:aacoxmal}
\begin{tabular}{cccc}
  Type & $n/2$ & $h$ & $h + r - 1$\\
  \hline
  $A_r$ & $r(r+1)/2$ & $r+1$ & $2r$\\
  $B_r$ & $r^2$ & $2r$ & $3r-1$\\
  $C_r$ & $r^2$ & $2r$ & $3r-1$\\
  $D_r$ & $r(r-1)$ & $2r-2$ & $3r-3$\\
  $E_6$ & $36$ & $12$ & $17$\\
  $E_7$ & $63$ & $12$ & $18$\\
  $E_8$ & $120$ & $30$ & $37$\\
  $F_4$ & $24$ & $12$ & $15$\\
  $G_2$ & $6$ & $6$ & $7$
\end{tabular}
\end{table}

\subsubsection{Infinitely many orbits}
\IM~\cite{ILRGLN} established that~$\Ab$ consists of infinitely many orbits,
using the \GRASSMANN\ variety associated with maximal nilpotent abelian
subalgebras in~$\germ{gl}_n$ for~$n\ge 6$. \MALCEV~\cite{MALCEV} classified abelian
subalgebras of maximal dimension in simple \LIE\ algebras, allowing us to
adapt the previous argument to show the
\begin{proposition}
\label{pr:sp:aainforbit}
The variety~$\Ab$ of abelian algebras consists of infinitely many
$G$-orbits if~$G$ has one of the following types:
\begin{equation*}
A_r \;(r \ge 5)\quad
B_r \;(r \ge 4)\quad
C_r \;(r \ge 5)\quad
D_r \;(r \ge 6)\quad
E_7\quad
E_8
.
\end{equation*}
\end{proposition}

\begin{proof}
Let~$r$ be the rank of our semi-simple group, $n$ its number of roots, $m$ the
largest dimension an abelian unipotent subalgebra of~$\Gg$ can have, and~$\Gm$
an abelian unipotent subalgebra with this dimension. Let~$\Gv$ be an abelian
subalgebra of~$\Gg$ contained in~$\Gm$. On the one hand, $\Gv$ lies in the
nilpotent cone of~$\Gg$, according to~(\ref{pr:sp:aanilporbit}) the~$G$-orbit
through~$\Gv$ in~$\Ab$ has dimension less than~$r - 2$. On the other hand, the
set of all possible~$\Gv$ is the~\GRASSMANN\ variety~$\Grass(r,\Gm)$, which
has dimension~$r(m-r)$. We conclude by an explicit computation that
whenever~$r(m-r) > r - 2$, the action of~$G$
on~$\Ab$ must have infinitely many orbits.
\end{proof}

\begin{theorem}[\MALCEV]
\label{th:sp:aamalcev}
Each simple \LIE\ algebra with the exclusion of~$B_4$, $D_4$ and~$G_2$ has up
to automorphisms only one commutative subalgebra of maximal dimension with
nilpotent elements. This dimension equals~$[\frac{1}{4}(r-1)^2]$ for the
algebra~$A_r\;(r>2)$ (brackets stand for the integer part of their argument), 
$\frac{1}{2}r(r-1) + 1$ for~$B_r\;(r>4)$, $\frac{1}{2}r(r+1)$ for~$C_n$,
$\frac{1}{2}r(r-1)$ for~$D_n$, and $16$, $29$, $36$, $9$ and~$5$ respectively
for~$E_6$, $E_7$, $E_8$, $F_4$ and~$B_3$. The algebra~$B_4$ has two classes of
conjugate abelian subalgebras of maximal dimension~$7$, $D_4$ has two classes
of dimension~$6$ and~$G_2$ has three classes of dimension~$3$.  
\end{theorem}

\subsubsection{Non algebraic anisotropic algebras}
While the irreducible component~$\Red$ of~$\Ab$ contains only algebraic
subalgebras of~$\Gg$~(\ref{co:sp:algred}), this is not the case of~$\Ab$.

\begin{proposition}
\label{pr:sp:aanonalg}
If a semi-simple group has rank large enough, its variety~$\Ab$ of abelian algebras
contains non algebraic elements.
\end{proposition}

\begin{proof}
Let~$G$ be a semi-simple group of rank~$r$, and~$s$ a degenerate semi-simple
element of its \LIE\ algebra~$\Gg$. If~$r$ is large enough, the set of nilpotent
elements in~$\Gg$ commuting with~$s$ contains a $r$-dimensional linear
subspace: this is most easily seen when~$s$ is the coroot associated with
an \emph{extremal} node in the \DYNKIN\ diagram of~$G$, in this case the
centralizer of~$s$ has semi-simple part a simple group of~rank~$r-1$ and one
can readily use \MALCEV\ theorem~(\ref{th:sp:aamalcev}). If the semi-simple part
of the centralizer of~$s$ in~$G$ has multiple simple ideals, one has to
apply~\MALCEV\ theorem on each of them, to conclude. Let~$m$ be a basis of
such a space, the \LIE\ algebra spanned by~$s +
m_1$, $m_2$, \dots,~$m_r$, is non algebraic, since it is not closed under
the \CJ\ decomposition.
\end{proof}


\appendix



\section{Tempered moving frames above degeneracy curves in~Grassmann varieties}
\label{se:sp:mf}
We study moving
frames along degeneracy curves in \GRASSMANN\ varieties and prove the
technical result~(\ref{th:sp:mfbasis}).
We use here the language of degeneracies introduced in~\ref{ss:sp:dgcompl}.

\subsubsection{Notations}
Let~$E$ be a vector
space of finite dimension, $A_1$ a $r$-dimensional subspace of~$E$ and~$C^o$ a
smooth and irreducible curve in~$\GGL(E)$. We let~$A_0$ be a $r$-dimensional
subspace lying in the
closure of the curve~$C^o A_1$ in the \GRASSMANN\ variety~$\Grass(r,E)$
of~$r$-dimensional subspaces of~$E$.
We denote
by~$\Tautological(r,E)$ the tautological bundle~$\Tautological(r,E) \to
\Grass(r,E)$.  The principal bundle~$\Basis(r,E) \to \Grass(r,E)$ whose fiber
at~$A$ is the set of all basis of~$A$ is a subbundle of the $r$-th bundle power
of~$\Tautological(r,E)$. It maps onto the principal bundle~$\Flag(r,E) \to
\Grass(r,E)$ whose fiber at~$A$ is the set of complete flags of~$A$.  If~$x$
is a basis of~$A$, we call its image in~$\Flag(r,E)$ the complete flag
corresponding to~$x$. Given a complete flag~$\Flag$ in~$A$, we say that a
basis~$x$ of~$A$ corresponds to~$\Flag$ when its image in~$\Flag(r,E)$
is~$\Flag$.

\subsubsection{Tempered moving frames along degeneracy curves}
Let~$A_1$ be a $r$-di\-men\-sion\-al subspace of~$E$ and $K$ an algebraic
group acting
linearly on~$E$. We choose a point~$A_0$ lying in the closure of the orbit~$K
A_1$ of~$K$ in~$\Grass(r,E)$ containing~$A_1$, a degeneracy curve~$C^o$ in~$K$
and~$c_0$ a point at infinity taking~$A_1$ to~$A_0$. If~$x$ is a basis
of~$A_1$, we obtain a moving frame along~$C^o A_1$ by sending~$c$ to the
point~$(cA_1, cx)$ of~$\Basis(r,E)\subset\Grass(r,E)\times E^r$.

This moving frame will usually not extend at~$c_0$. In the first place, it
will not do
because the vectors~$cx_k$ may become infinitely small or infinitely
large. In the second place, the lines defined by two
vectors~$cx_i$ and~$cx_j$ may degenerate to a single line. If~$cx_i$
is asymptotically smaller than~$cx_j$, no linear combination of~$cx_i$
and~$cx_j$ can yield a second direction at~$c_0$.

The first difficulty is circumvented by replacing~$cx$ by the
associated tempered moving frame: it is obtained from~$cx$ by rescaling its
terms~(\ref{de:sp:tempered}). We are able to characterize the basis
in~$A_1$ for which the second difficulty does not occur.
They are the basis of~$A_1$ whose associated
flag is a refinement of some partial flag in~$A_1$ determined by~$C^o$. We
call this partial flag the~\emph{magnitude orders flag}, it is defined
in~(\ref{ss:sp:mfflag}).
\medbreak

From now on we denote by~$R_0$ the local ring of the smooth completion~$C$
of~$C^o$ at~$c_0$, by~$\Gm_0$ its maximal ideal, by~$L_0$ its ring of
fractions and~by~$\epsilon\in\Gm_0$ a local parameter of~$C$
at~$c_0$. As the irreducible smooth curve~$C^o$ uniquely determines its smooth
completion~$C$, we call~$R_0$ the local ring of~$C^o$ at~$c_0$ and~$\epsilon$
a local parameter of~$C^o$ at~$c_o$.

\subsubsection{Filtration by the magnitude order}
We consider a linear representation~$V$ of~$K$, for instance an
exterior power of~$E$. For any vector~$v\in
V$, the rational map~$c v$ from~$C$ to~$V$ is an element of~$L_0\otimes
V$. Thus, the asymptotic behaviour of vectors in~$V$
under the operation of~$c$ near~$c_0$ is related to the structure of
the~$R_0$-module~$L_0 \otimes V$. 

The $L_0$-vector space~$L_0\otimes V$ is a $R_0$-module filtered by
the submodules~$\Gm_0^k (R_0\otimes V)$ for~$k \in \Z$. For any~$v\in
L_0\otimes V\setminus\set{0}$
we call the number
\[
\omega(v) = \sup\set{k \in \Z \mid \Gm_0^k (R_0\otimes V) \ni v }
\]
the \emph{magnitude order} of~$v$ near~$c_0$, or briefly the magnitude order
of~$v$. Note that the zero vector has no magnitude order. For
any~$v_1$, $v_2$ in~$L_0\otimes V\setminus\set{0}$ we have
\[
\omega(v_1 + v_2)
\ge
\min\set{\omega(v_1),\omega(v_2)}
,
\]
unless~$v_1 + v_2 = 0$.

\subsubsection{Magnitude orders flag and corresponding basis}
\label{ss:sp:mfflag}
We describe a filtration on~$A_1$ associated to the filtration
of~$L_0\otimes E$ by the magnitude order, which yields a partial flag
in~$A_1$. In the theory of the module $L_0 \otimes E$ over the
principal ideal domain~$R_0$, this flag is related to the invariant
factors of the submodule~$R_0 \otimes A_1$ of the finitely generated
module~$R_0 \otimes E$.

To any integer~$k$ we associate the subspace~$(A_1)_k$ of~$A_1$ defined by
\[
  (A_1)_k = \set{
    x \in A_1\setminus\set{0} \mid \omega(cx) \ge k 
  } \cup\set{0}
  .
\]
Let~$(a_1, \dots, a_m)$ be the increasing sequence of integers~$k$ such
that~$(A_1)_k$ is a strict superset of~$(A_1)_{k+1}$. We thus
have~$(A_1)_{a_1} = A_1$ and~$(A_1)_{1 + a_m} = \set{0}$, and it is
convenient to put~$a_{m+1} = 1 + a_m$. This yields a partial flag
\[
\set{
(A_1)_{a_1}
\supset 
\cdots
\supset
(A_1)_{a_m}
}
\]
of~$A_1$.

\begin{definition}
The partial flag~$\Flag_1^\omega = \set{
(A_1)_{a_1}
\supset 
\cdots
\supset
(A_1)_{a_m}
}$ of~$A_1$ is the \emph{magnitude orders flag}.
A \emph{magnitude orders basis} of~$A_1$ is a basis~$x\in A_1^r$ of~$A_1$
corresponding to a complete flag finer than the magnitude orders flag.
\end{definition}

The natural map~$\Basis(r,E) \to \Flag(r,E)$ is~onto, which implies
the existence of magnitude orders basis of~$A_1$.

\begin{definition}
\label{de:sp:tempered}
Let~$x$ be basis of~$A_1$. The tempered moving frame associated to~$x$ is the
map sending~$c\in C^o$ to~$(cA_1, \epsilon^{-\omega(cx_1)} cx_1, \dots,
\epsilon^{-\omega(cx_r)} cx_r)$.
\end{definition}

\subsubsection{Extension of tempered moving frames}
Our characterization of tempered moving frames that extend at infinity is the
following

\begin{theorem}
\label{th:sp:mfbasis}
Let~$x$ be a basis of~$A_1$. The following
conditions are equivalent:
\begin{enumerate}
\item
$x$ is a magnitude orders basis of~$A_1$;
\item
the finite sequence~$(\omega(cx_k))_{1\le k \le r}$ is non-decreasing and
for any~$1\le k \le r$ we have
\[
  \omega(cx_1 \wedge \cdots \wedge cx_k) 
  =
  \sum_{i=1}^k \omega(cx_i)
\]
for magnitude orders in the~$R_0$-module $\Alt^k(L_0\otimes E)$.
\item
the tempered moving frame associated to~$x$
extends at~$A_0 = c_0 A_1 $, and its value at~$c_0$ is a basis of~$A_0$.
\end{enumerate}
\end{theorem}

We will only give a few indications about its proof, for it is rather lengthy
and has little to do with the geometry of varieties of reductions. The hard
work is the proof that~(2) implies~(3). This can be done by using a variation
of Smith's algorithm computing the normal form of a matrix whose coefficients
are in the principal ideal domain~$R_0$. From this standpoint~(2) means
that the principal minors of the matrix whose columns are the coordinates of
the tempered frame associated to~$x$ are invertible in~$R_0$.

We also state without proof the following technical result:

\begin{proposition}
\label{pr:sp:mfdeco}
Let~$A_1 = B^1 \oplus \cdots\oplus B^l$ be a direct sum decomposition
of~$A_1$. There exists a magnitude order basis of~$A_1$ whose terms are
in~$B^1 \cup \cdots \cup B^l$.
\end{proposition}





\section{Closures of decomposition classes}

We use the rigidity theorem~\ref{co:sp:ar} to show that the closure of a
decomposition class is a union of decomposition class. It seems that this was
yet only proved for the reductive symmetric pairs associated to the Cartesian
square of a reductive group~\cite[39.2.7]{TY}.

\begin{theorem}
\label{th:sp:dcclosure}
Let~$(G,\theta)$ be a reductive symmetric pair, and~$\Gp$ its anisotropic
space. For any~$x$ in~$\Gp$, the closure of the decomposition
class~$\Decomposition(x)$ of~$x$ in~$\Gp$ is a union of decomposition classes.
\end{theorem}

\begin{proof}
Let~$K = (G^\theta)^o$ be the connected component of the fixed point group
of~$\theta$.  According to~\cite[39.5.2]{TY} we have:
\[
\Decomposition(x)
 = K \cdot (\dcentralizer(x_s)_o + x_n) 
 = K \cdot (\dcentralizer(x)_o) 
\]
(see~(\ref{eq:sp:dcdcentgen}) for the definition of~$\dcentralizer(x_s)_o$
and~$\dcentralizer(x)_o$). Let~$y$ be a point in the
closure~$\bar\Decomposition(x)$ of~$\Decomposition(x)$ in~$\Gp$. We show
that~$\Decomposition(y)$ is contained in~$\bar\Decomposition(x)$. Since this
closure is~$K$-stable, it is enough to prove that~$\dcentralizer(y_s) + y_n
\subset \bar\Decomposition(x)$.

Let~$\Gamma^o \subset K \times \dcentralizer(x)_o$ be a smooth curve
such that~$y$ lies in the closure of the image of~$\Gamma^o$ by the map
sending~$(g,z)$ to~$gz$. For all~$(g,z)$ in~$\Gamma^o$ we
have~$\dcentralizer(gz) = g\dcentralizer(x)$ so that the projection~$C^o$
of~$\Gamma^o$ in~$K$ pushes the anisotropic algebra~$\dcentralizer(x)$ toward
an anisotropic algebra containing~$y$. A double centralizer is closed
under~\CJ\ decomposition, hence we can assume by the rigidity
theorem~\ref{co:sp:ar} that~$y_s$ belongs to~$\dcentralizer(x)$
and~$C^o$ is a subset of the centralizer of~$y_s$
in~$K$. But then~$\dcentralizer(y_s) \subset \dcentralizer(x_s) \subset
\dcentralizer(x)$ so we are done.
\end{proof}


\bibliographystyle{ieeetr}
\bibliography{redgen}
\end{document}